\documentclass{amsart}
\usepackage{amssymb}
\title{Classification of subgroups isomorphic to $\PSL_2(27)$ in the Monster}
\author{
Robert A.~Wilson}
\address{
School of Mathematical Sciences\\
Queen Mary, University of London\\
Mile End Road\\
London E1 4NS}
\date{
19th April 2013}
\begin{document}

\def\FF{{\mathbb F}}
\def\squareforqed{\hbox{\rlap{$\sqcap$}$\sqcup$}}
\def\qed{\ifmmode\squareforqed\else{\unskip\nobreak\hfil
\penalty50\hskip1em\null\nobreak\hfil\squareforqed
\parfillskip=0pt\finalhyphendemerits=0\endgraf}\fi\medskip}

\newcommand{\im}{\mathrm{im}\,}
\newcommand{\udot}{{}^{\textstyle .}}
\newcommand{\diag}{\mathrm{diag}}
\newcommand{\PSL}{\mathrm{PSL}}
\newcommand{\Mat}{\mathrm{M}}
\newcommand{\Aut}{\mathrm{Aut}\,}
\newcommand{\Sz}{\mathrm{Sz}}
\newcommand{\PSU}{\mathrm{PSU}}
\renewcommand{\L}{\mathrm{L}}
\newcommand{\U}{\mathrm{U}}
\newcommand{\Suz}{\mathrm{Suz}}
\newcommand{\Co}{\mathrm{Co}}
\renewcommand{\Mat}{\mathrm{M}}

\newtheorem{theorem}{Theorem}

\maketitle

\begin{abstract}
As part of the problem of the determination of the maximal subgroups of the Monster
we 
classify subgroups isomorphic to $\PSL_2(27)$.
Indeed, we prove that the Monster does not contain any subgroup
isomorphic to $\PSL_2(27)$.
\end{abstract}

\section{Introduction}
\label{intro}
The Monster is the largest of the $26$ sporadic simple groups. The maximal subgroups
of the other $25$ are all known, so it would be satisfying to complete this project
also for the Monster.
The problem of determining the maximal subgroups of the Monster has a long
history (see for example \cite{Anatomy1,Anatomy2,oddlocals,J1,post,L259,A5subs,S4subs,L241}).
The cases left open by previous published work are
normalizers of simple subgroups with trivial centralizer, and isomorphic to one of
the groups
\medskip

\centerline{$\PSL_2(8)$, $\PSL_2(13)$, $\PSL_2(16)$, $\PSL_2(27)$, $\PSU_3(4)$, $\PSU_3(8)$, $\Sz(8)$.}\medskip
Of these, $\PSL_2(8)$ and $\PSL_2(16)$ have been classified in unpublished work
of Holmes. The cases $\PSL_2(27)$ and $\Sz(8)$ are
particularly interesting because no subgroup isomorphic to $\PSL_2(27)$ or $\Sz(8)$
 is known. Here we consider the case $\PSL_2(27)$, and we show that in fact there
is no subgroup isomorphic to $\PSL_2(27)$ in the Monster.

\section{Theoretical results}
\label{theory}
The strategy we use here is the standard one for $\PSL_2(q)$, namely
to classify the possibilities for the $BN$-pair, consisting 
in this case of $B\cong 3^3{:}13$ and $N\cong D_{26}$ intersecting in
the `torus' of order $13$.

First we use the $3$-local analysis to classify subgroups of the Monster
isomorphic to $3^3{:}13$. Since neither $3A$-elements nor $3C$-elements
can form a pure $3^3$ (see \cite{oddlocals}), the $3$-elements in any $3^3{:}13$
must be in
class $3B$, and we know from \cite{oddlocals} that there are just three classes of pure
$3B$-type $3^2$. Clearly all the $3^2$ subgroups of our $3^3$ must be of the same type.

Consider first the case when they are of type $3B_4(i)$ in the notation of \cite{oddlocals}.
Such a $3^2$ has centralizer $3^2.3^5.3^{10}.\Mat_{11}$. It is shown in Theorem 6.5 of \cite{oddlocals}
that there is a unique conjugacy class of such $3^3$. The normalizer of this $3^3$ has the
shape $3^3.3^2.3^6.3^6.(\PSL_3(3)\times SD_{16})$. Hence there is, up to conjugacy, a unique
group $3^3{:}13$ of this type. It has centralizer $3^2{:}D_{8}$, and its normalizer is
a group of shape
$$(3^3{:}(2\times 13{:}3) 
\times 3^2{:}SD_{16}).\frac12=(3^3{:}13{:}3\times 3^2{:}D_8).2,$$ 
that is, a subgroup of index $2$ in
$3^3{:}(2\times 13{:}3)\times 3^2{:}SD_{16}$.

Now this $3^3{:}13$ contains $13A$-elements, and the invertilizer of a $13A$-element
has shape $13{:}2 \times \PSL_3(3)$. Hence there are just 118 copies of $D_{26}$
containing a given element of class $13A$. 
Elementary calculations show that the subgroup $3^2{:}D_8$ of $\PSL_3(3)$ has orbits
of sizes $9+6+6+12+12+36+36$ on the $117$ involutions in $\PSL_3(3)$, and in particular
has no regular orbit.
It follows that any $\PSL_2(27)$ of this type
would have non-trivial centralizer. This is a contradiction, as none of the element-centralizers
in the Monster contains $\PSL_2(27)$.

Next consider the cases $3B_4(ii)$ and $3B_4(iii)$. By Propositions 6.1 and 6.2 of \cite{oddlocals},
the whole of the $3^3$ lies inside a unique $3^{1+12}$. Therefore $3^3{:}13$ lies inside
$3^{1+12}\udot2\udot \Suz{:}2$. Now in $\Suz{:}2$ the Sylow $13$-normalizer has the shape $13{:}12$.
Therefore the $13$-element normalizes just four subgroups of shape $3^3$ in the $3^{1+12}$,
and these are permuted by the $13$-normalizer. It follows that there is a unique class of $3^3{:}13$
of this type in the Monster. Such a subgroup has centralizer of order $3$, and normalizer of shape
$3 \times 3^3{:}(2\times 13{:}3)$.

To summarise the results of this section, we have proved the following.
\begin{theorem}
\label{3313theorem}
There are exactly two conjugacy classes of subgroups isomorphic to  $3^3{:}13$ in the Monster, one of
which contains $13A$-elements, while the other contains $13B$-elements.
\end{theorem}
\begin{theorem}
\label{13Atheorem}
There is no $\PSL_2(27)$ in the Monster containing $13A$-elements.
\end{theorem}
\section{Computational strategy}
\label{strategy}
At this point we resorted to computer calculations to finish the job. The original calculations
were done about ten years ago, using the mod $2$ construction 
of the Monster \cite{Mmod2}, but subsequently
lost. The calculations were therefore repeated, as described here, using the mod $3$
construction \cite{2loccon}. The general methods of computation are described in
\cite{2loccon,post,A5subs}, and summarized in \cite{L241}, which also contains some
improved methods. As in these references, we take $a,b$ as generators for the 
subgroup $2^{1+24}\udot\Co_1$, and $T$ as a `triality element', cycling the $3$
central involutions in a subgroup $2^2.2^{11}.2^{22}.\Mat_{24}$ of 
$2^{1+24}\udot\Co_1$.

By Theorems~\ref{3313theorem} and \ref{13Atheorem} we have reduced to
the case in which $3^3{:}13$ contains $13B$-elements. The invertilizer 
of a $13B$-element is
$13^{1+2}{:}4A_4$, and therefore
there are exactly 78 copies of $D_{26}$ containing a given
$13B$-element.

We break the calculations down into a number of steps. First we make the part of the
$13$-normalizer that we can easily find inside $2^{1+24}\udot\Co_1$. This is done in 
Section~\ref{13.12A4}, where we obtain a group $13{:}(3 \times 4A_4)\cong
(13{:}3\times 2A_4){:}2$. Then in Section~\ref{chpost1}
we pick a non-central involution in this group, and find an element of the Monster
conjugating it to the central involution. This allows us in Section~\ref{findsecond13}
to find another element of order $13$ commuting with the first one, thereby
extending the subgroup to $13^{1+2}{:}(3\times 4A_4)$.
This group contains all the involutions which extend $13$ to $D_{26}$.

Then in Section~\ref{find27} we find an element of order $3$ which, together
with our original element of order $13$, generates $3^3{:}13$.
Finally, in Section~\ref{finale} we 
show that under the action of the normalizer of $3^3{:}13$ the $78$ ways of extending
$13$ to $D_{26}$ fall into six orbits. Then we test the resulting $6$
 cases to see if the
given element of order $3$ in $3^3{:}13$, multiplied by one of the $13$ involutions
in $D_{26}$, gives an element of order $3$. This is a criterion which distinguishes
$\PSL_2(27)$ from all other groups.

\section{Finding $(13{:}3\times 2A_4){:}2$}
\label{13.12A4}
The strategy here is to work first in $\Co_1$, to find enough of the centralizer of a
$2B$-element to obtain a group $2^2\times G_2(4)$. Then we conjugate one of the
central involutions to the other,
in such a way that we obtain 
$A_4\times G_2(4)$. Within this subgroup
we find a copy of $\PSL_2(13)$ by random search, and then a copy of $13{:}6$
inside $\PSL_2(13)$. Finally we apply the standard method known 
colloquially as `applying the formula' in order to lift to elements 
of $2^{1+24}\udot\Co_1$ which normalize
a particular element of order $13$ we choose.

\subsection{Constructing $A_4\times G_2(4)$ in $\Co_1$}
\label{mkA4G24}
We take $a,b$ to be the original pair of generators of $2^{1+24}\udot \Co_1$,
and first work in the quotient $\Co_1$ to make the element
$$c_1=(ab)^4(ab^2)^3$$
of order $26$, so that 
$$i_1=(c_1)^{13}$$
is an element of class $2B$ in $\Co_1$. We make
$$c_2=abi_1[ab,i_1]^5$$
which centralizes $i_1$,
and let $$i_2=(c_2)^{13}.$$
The elements $c_1,c_2$ then generate $2^2\times G_2(4)$, in which the
central $2^2$ is generated by $i_1,i_2$.
Then let 
\begin{eqnarray*}
n_1&=& (ai_1)^5(ab)^{-2}i_2(ab)^2a(ab)^{-2}\cr
n_2&=&(ai_1)^5(i_1i_2a)^5
\end{eqnarray*}
to give elements which normalize the $2^2$ and give us a group $A_4\times G_2(4)$.

The normal subgroup $A_4$ is generated by 
\begin{eqnarray*}
a_1&=&i_1\cr 
a_2&=&(n_1n_2)^{13},
\end{eqnarray*}
while the normal $G_2(4)$ is generated by 
\begin{eqnarray*}
g_1&=&(c_1)^2\cr
g_2&=&(n_1n_2)^{3}.
\end{eqnarray*}

\subsection{Constructing a subgroup $A_4\times 13{:}6$}
\label{G24-78}
We then make standard generators of $G_2(4)$, as defined in
\cite{webatlas}, as
\begin{eqnarray*}
g_3&=&(g_1^4g_2)^4\cr
g_4&=&((g_1g_2g_1g_2^2)^3)^{g_2^4}
\end{eqnarray*}
and generators for a subgroup $\PSL_2(13)$ 
can then be read off from \cite{webatlas} as
\begin{eqnarray*}
g_5&=&((g_3g_4)^3g_4)^3((g_3g_4)^4g_4g_3g_4(g_3g_4^2)^2)^3((g_3g_4)^3g_4)^{-3}\cr
g_6&=&(g_3g_4g_3g_4^2)^{-2}(g_3g_4(g_3g_4g_3g_4^2)^2)^5(g_3g_4g_3g_4^2)^2
\end{eqnarray*}
Inside here we find that a subgroup $13{:}6$ is generated by $g_5$ and
$$g_7= (g_6)^{g_5g_6^2},$$
and we may take the element of order $13$ to be
$$g_8=[g_5,g_7].$$
\subsection{Lifting to $2^{1+24}\udot\Co_1$}
\label{form78A4}
Now we `apply the formula' to lift to $2^{1+24}\udot\Co_1$. 
That is, we replace the elements $a_1,a_2,g_5,g_7$ by new elements, in the same cosets
of $2^{1+24}$, which normalize the subgroup $\langle g_8\rangle$ of order $13$.
Specifically, we make
\begin{eqnarray*}
a_1'&=&g_8a_1(g_8a_1^{-1}g_8a_1)^6\cr
a_2'&=&(g_8a_2(g_8a_2^{-1}g_8a_2)^6)^2\cr
g_5'&=&g_8g_5(g_8^{-1}g_5^{-1}g_8g_5)^6\cr
g_7'&=&(g_8g_7(g_8^9g_7^{-1}g_8g_7)^6)^2
\end{eqnarray*}
so that $a_1',a_2'$ generate $2A_4$ and $g_5',g_7'$ generate 
$26\udot6=(2\times 13{:}3)\udot 2)$,
commuting with each other. Thus they together generate
$$2\udot(A_4\times 13{:}6).$$
An element of order $12$ normalizing $\langle g_8\rangle$ 
and commuting with $\langle a_1',a_2'\rangle\cong 2A_4$ and with $g_5'$ may be obtained as
$$g_9=g_5'g_8^3g_7'g_8^{10}.$$

\section{Changing post}
\label{chpost1}
The process of `changing post' really consists of two parts. The first part consists of
finding a word $x$ in the generators of the Monster, which conjugates a given involution
in $C(z)$, to $z$. This part is more or less algorithmic, especially as, in this case,
we already in \cite{L241}
found a word conjugating an involution in the desired $C(z)$-conjugacy class, to $z$.
Here the involution which we want to map to $z$ is $a_1'g_5'$.

The second part consists of `shortening the word' for an element
$g^x$, where both $g$ and $g^x$ lie in $C(z)$. This part is more \emph{ad hoc}, and 
involves often quite laborious search for a word in $a$ and $b$ which is equal to the
desired element. 
In this section, the element we want to write as a word in $a$ and $b$
is the appropriate conjugate of the element $g_9$ of order $12$. Our strategy is to
first find a word for its image in $\Co_1$, and then to lift to $2^{1+24}\udot\Co_1$.
Even within $\Co_1$, the search is not easy, and we perform it in stages,
first dealing with the involution which is its sixth power, and then its fourth power,
before finally reaching the element itself. In the course of these calculations,
we also identify two useful elements which centralize the given element of
order $12$.

\subsection{Conjugating the involution $a_1'g_5'$ to $z$}
\label{changepost}
The next step in extending to the full $13$-normalizer $13^{1+2}{:}(3\times 4S_4)$
is to conjugate the involution $i_3=a_1'g_5'$ to $z$. First we make our `standard'
involution in this conjugacy class in $2^{1+24}\udot\Co_1$ as follows.
As in \cite{L241} we make
\begin{eqnarray*}
h&=&(ab)^{34}(abab^2)^3(ab)^6\cr
i&=&(ab^2)^{35}((ababab^2)^2ab)^4(ab^2)^5\cr
k_1&=&hihi^2\cr
k_2&=&hihihi^2\cr
k&=&(k_1k_2)^3k_2k_1k_2
\end{eqnarray*}
Then we make
$$k'=((a^2)^{(ab)^3}k^8)^{11}k^{11}$$
as our standard involution in this conjugacy class. This element is carefully
chosen so that $T^{-1}k'T$ is an element of the normal $2^{1+24}$.

We calculate once and for all how to conjugate this element to $z$. This calculation
was already done in \cite{L241}, and the result is that
if
$$k_3=(ab)^3(ab^2)^{20}(ababab^2abab^2)^8(ababab^2ab)^{12}(ababab^2)^5$$
then $$(k')^{Tk_3T}=z.$$

It remains now to conjugate $i_3$ to $k'$.
Now if two elements of $\Co_1$-class $2C$ have product of order $13$ or $35$,
then this product is fixed-point-free in its action on $2^{24}=2^{1+24}/2$, and
hence when we lift to $2^{24}\udot\Co_1$ the product remains of odd order. Thus we
can conjugate one to the other in $2^{24}\udot\Co_1$ using the standard formula.
Lifting to $2^{1+24}\udot\Co_1$ is then easy.
So we search for conjugates of $i_3$ and $k'$ whose product has order $13$ or $35$,
and thereby find that if
$$l_3=(ab^2)^4(k'(i_3)^{(ab^2)^4})^6$$
then $l_3$ conjugates $i_3$ to $k'$, and therefore $l_3Tk_3T$ conjugates $i_3$ to $z$.

\subsection{Finding the centralizer of $(g_9')^2$ in $\Co_1$}
\label{mk3ah}
This conjugation takes the element $g_9$ to an element 
$$g_9'=g_9^{l_3Tk_3T}$$
which has order $12$ in the quotient $\Co_1$. We now want to find this element
as a word in $a,b$, so as to eliminate the occurrences of $T$.
This  is by no means a simple process. In this subsection we obtain a word
for an element which is congruent to $(g_9')^2$ modulo $2^{1+24}$.

First note that $g_9^6=z$, so that $(g_9')^6$ is (modulo $2^{1+24}$)
in the normal $2^{11}$ subgroup of the standard copy $\langle h,i\rangle$ of
$2^{11}{:}\Mat_{24}$. By a random search we find a subgroup $2^{11}{:}\Mat_{12}$
centralizing $(g_9')^6$, generated by
\begin{eqnarray*}
t_1&=&(i^2)^{(hi^2)^6(hi)^{-6}}\cr
t_2&=&(i^2)^{(hi^2)^5(hi)^{-13}}\cr
t_3&=&(i^2)^{(hihihi^2hi)^3(hi)^{-2}}
\end{eqnarray*}
Moreover, the central involution of this group is 
$$t_0=(t_1t_3t_1t_3t_1t_3^2)^{11}.$$

Then we conduct another random search in this subgroup for elements
commuting with the element $(g_9')^4$ of order $3$. Writing
\begin{eqnarray*}
u&=&t_1\cr
v&=&t_2t_3\cr
t_4&=&((uv)^4)^{(uvuvuv^2uvuv^2)^9(uvuvuv^2uv)^{10}}\cr
t_5&=&((uv)^4)^{(uvuvuv^2)^{10}(uvuvuv^2uvuv^2)^2}\cr
t_6&=&(t_4)^{t_4t_5t_4t_5t_4t_5^2(t_4t_5)^7}
\end{eqnarray*}
we have that $t_6$ is in fact the inverse of $(g_9')^4$, modulo $2^{1+24}$.
\subsection{Finding the centralizer of $g_9'$ in $\Co_1$}
\label{mk3ij}
Working first in the $\Mat_{12}$ quotient of $\langle t_1, t_2t_3\rangle$ we find
that the following elements commute with $t_6$ modulo the $2$-group:
\begin{eqnarray*}
t_8&=&(t_1)^{((t_1t_2t_3)^3t_2t_3)^3(t_2t_3)^5}\cr
t_9&=&(t_1)^{((t_1t_2t_3)^3t_2t_3(t_1t_2t_3)^2t_2t_3)^7(t_2t_3)^6}
\end{eqnarray*}
Applying the formula we obtain
\begin{eqnarray*}
t_8'&=& t_6^2t_8t_6^2t_8^2t_6^2t_8\cr
t_9'&=& t_6^2t_9t_6^2t_9^2t_6^2t_9
\end{eqnarray*}
and then
\begin{eqnarray*}
t_{10}&=& (t_9't_8't_9')^3
\end{eqnarray*}
is congruent to $(g_9')^3$, modulo $2^{1+24}$. We also make
some elements commuting with $g_9'$ modulo the $2$-group, as follows:
\begin{eqnarray*}
t_{11}&=& ((t_1t_3)^8t_6^2)^3t_8'(t_9')^2t_{10}\cr
t_{12}&=& (t_8')^2((t_1t_3)^8t_6^2)^3(t_8')^2(t_9')^2
\end{eqnarray*}

\subsection{Lifting to $2^{1+24}\udot\Co_1$}
\label{liftg9}
Now we know that $t_{10}t_6$ is congruent modulo $2^{1+24}$ to the
inverse of $(g_9')^{l_3Tk_3T}$. It remains to find the correct element of $2^{1+24}$
to multiply by. Using the method explained in \cite{L241}, we obtain the element
$$w=t_{10}t_6p_1d_3d_4d_5d_7d_8d_9d_{10}d_{11}d_1p_3p_4p_6p_7p_8.$$

We lift the elements $t_{11},t_{12}$ to elements which commute with $w$,
by the following method.
First apply the formula, to get elements $t_{11}', t_{12}'$ which commute
with $w^4$:
\begin{eqnarray*}
t_{11}'&=&w^4t_{11}w^4t_{11}^{-1}w^4t_{11}\cr
t_{12}'&=& w^4t_{12}w^4t_{12}^{-1}w^4t_{12} 
\end{eqnarray*}
Then make the part of $2^{1+24}$ which commutes with $w^4$: 
by computing $w^4$ in the $24$-dimensional $\mathbb F_2$-representation
of $\Co_1$, we find that this is generated by
\begin{eqnarray*}
q_1&=&d_9d_{12}\cr
q_2&=&d_1d_4d_5d_6d_{10}d_{11}\cr
q_3&=&d_3d_5d_7d_{10}d_{12}\cr
q_4&=&d_2d_6d_7d_8d_{10}d_{12}\cr
q_5&=&p_4p_6p_9p_{12}d_5d_8d_{10}\cr
q_6&=&p_3p_6p_7p_9d_4d_7d_{12}\cr
q_7&=&p_2p_5p_6p_7p_{10}p_{12}d_7d_{11}\cr
q_8&=&p_1p_7p_8p_9p_{10}p_{12}d_4d_8
\end{eqnarray*}
where $p_1,\ldots,p_{12},d_1,\ldots,d_{12}$ are the generators for
$2^{1+24}$ given in \cite{L241}.
Finally test all multiples of $t_{11}'$ and $t_{12}'$ by products of the $q_i$.
We find the following elements which commute with $w$:
\begin{eqnarray*}
t_{11}''&=&q_5q_6q_8t_{11}'\cr
t_{12}''&=&q_4q_5q_6q_7t_{12}'
\end{eqnarray*}
Note also that $w$ commutes with $q_2q_3q_4$, and modulo the central involution, also with
$q_4$.

\section{Finding the full $13B$-centralizer}
\label{findsecond13}
In order to extend $(13{:}3\times 2A_4){:}2$ to
$13^{1+2}{:}(3\times 2A_4){:}2$,
we now seek an element of order $13$ which is normalized by $w$. First we work
in the quotient $\Co_1$, and afterwards lift to $2^{1+24}\udot\Co_1$.

\subsection{Extending $12$ to $13{:}12$ in $\Co_1$}
\label{find13}
Now the element $t_{11}$ maps to a $2B$-involution in the quotient $\Co_1$,
and the element of order $13$ we are looking for centralizes \emph{either}
this involution, \emph{or} $t_0t_{11}$. But conjugating by $t_{12}$
interchanges these two cases,
so we can assume the former.
We therefore begin by making the centralizer of $t_{11}$ in the quotient $\Co_1$.
Let
\begin{eqnarray*}
h_5&=& ((at_{11})^6t_0t_6)^4\cr
h_6&=& (t_0t_6(at_{11})^6)^4
\end{eqnarray*}
which are elements of order $21$ generating 
$G_2(4)$ in this centralizer. We search for a subgroup $\PSL_2(13)$
containing $t_6$, and find that $\langle h_7,t_6\rangle$ is such a subgroup, where
$$h_7= ((h_5h_6h_5h_6^2)^5) ^{h_5h_6^7}.$$
Inside this copy of $\PSL_2(13)$, we find an element of order $13$
$$h_8'= (h_7t_6^2)^4t_6^2h_7t_6^2(h_7t_6^4)^2$$
and the one normalized by $t_6$ is
$$h_8= (h_8')^{h_7(h_8')^{10}}.$$

Then we work with the centralizer of $t_6$ in $G_2(4)$
to conjugate this $13$-element to one which
is normalized also by $t_{10}$. We first make this centralizer by a random
search through $3A$-elements of $G_2(4)$ to find some which commute.
We found
\begin{eqnarray*}
h_{10} &=& ((h_5h_6h_5h_6h_5h_6^2h_5h_6)^7)^{h_5^{18}h_6^{10}}\cr
h_{11} &=& ((h_5h_6h_5h_6h_5h_6^2h_5h_6)^7)^{h_5^{18}h_6^{15}}
\end{eqnarray*}
which generate $A_5$. Conjugating by random elements of this centralizer
we quickly find one of the $13$-elements we are looking for, namely
$$h_{12}=(h_{10}h_{11}^2h_{10})^4 h_8h_{10}h_{11}^2h_{10}.$$

\subsection{Lifting}
\label{lift13}
The main lifting problem is to lift the element of order $13$ to one which is
normalized by $w$. 
Since there are $2^{24}$ elements of order $13$ in the given coset of $2^{1+24}$,
only two of which are normalized by $w$, a brute force search is out of the
question (or at least, unwieldy).
We therefore do this in two stages, first finding an appropriate
conjugating element to get a $13$-element inverted by $w^6$. Since $w^6$
centralizes just $2^{12}$ out of the $2^{24}$ factor, this divides the problem
into two searches, each in a population of size $2^{12}$.

First we work in the $24$-dimensional $\mathbb F_2$-representation of $\Co_1$
and find that the fixed space of $w^6$ is spanned by vectors which lift to
the following elements of $2^{1+24}$:
all even products of the $d_i$, together with $$d_1p_1p_3p_6p_8p_{10}p_{12}.$$
Hence in the first search we may test conjugates just by the $p_i$. We find that
the correct conjugating element is $$p_1p_3p_5p_6p_7p_{12}.$$

In the second search we test conjugates by $d_id_{i+1}$ and
 $d_1p_1p_3p_6p_8p_{10}p_{12}$. We find that exactly two conjugating elements work:
\begin{eqnarray*}
&&d_1d_2d_3d_5d_8d_9d_{11}p_1p_3p_6p_8p_{10}p_{12}\cr
&&d_1d_3d_5d_6d_7d_9d_{10}d_{11}d_{12}p_1p_3p_6p_8p_{10}p_{12}.
\end{eqnarray*}
Let $h_{12}'$ and $h_{12}''$ be the
respective conjugates of $h_{12}$.

\subsection{Finding the $12$-normalizer}
\label{norm12}
In order to obtain all possible $13$-elements normalized by our element of order $12$,
we need to conjugate not just by elements of its centralizer, but by elements of
its normalizer. Indeed it turns out that we need to replace $w$ by its $7$th power.

Such a normalizing element can be found inside the centralizer of $w^3$ as follows.
First we conjugate $a_1'$ to $t_{11}'$, by conjugating by $(t_{11}'a_1')^2$, so that
we have the full $A_4\times G_2(4)$ available. In particular, the element
$$(t_{11}'a_1')^3(a_2'^2a_1'a_2')(t_{11}'a_1')^2$$
is an involution in the $A_4$, but not equal to $t_{11}'$.
(In the end, we found we did not need to use this element.)

Within the $A_5$ generated by $h_{10}, h_{11}$ we make the $15$ involutions
as conjugates, by powers of $h_{10}h_{11}$, of $(h_{10}h_{11})^2h_{11}$ 
and $h_{11}h_{10}h_{11}^2h_{10}$ and their product. We find that
the first involution conjugated by $(h_{10}h_{11})^2$ commutes with $w$,
and therefore 
the normalizing element we want is (modulo the $2$-group)
$$t_{14}=(t_{11}'a_1')^3(a_2'^2a_1'a_2')(t_{11}'a_1')^2
(h_{10}h_{11})^3(h_{11}h_{10}h_{11}^2h_{10})(h_{10}h_{11})^2.$$

To lift to $2^{1+24}\udot\Co_1$, 
we first apply the formula to get an element which commutes with $w^4$:
$$t_{14}'= w^4t_{14}w^4t_{14}^3w^4t_{14}.$$
Finally multiplying by combinations of the $q_i$ we find that the element
we want can be taken to be
$$t_{14}''=q_3q_7q_8t_{14}'.$$

We also made an element $t_{13}'$ which conjugates $w$ to its 5th power, but this
turned out not to be necessary.
 
\subsection{Testing commuting with the first $13$-element}
\label{test13s}
We are aiming to find the normalizer of $g_8$, so we have to test our candidate
elements of order $13$ to see which one(s) commute with $g_8^{l_3Tk_3T}$.
The candidate elements are conjugates of $h_{12}'$ and $h_{12}''$ by
combinations of $t_{11}'', t_{12}'', t_{13}', t_{14}''$.
Of these, we found that
the one which works is
$$w_1=(h_{12}'')^{t_{11}''t_{14}''}.$$

We now have generators for $13^{1+2}{:}(3\times 4A_4)$. These are
best taken as the generators $a_1', a_2', g_5', g_7'$ given above, together with
the conjugate of $w_1$ by $(l_3Tk_3T)^{-1}$, that is
$$w_1'=l_3Tk_3Tw_1T^{-1}k_3^{-1}T^{-1}l_3^{-1}.$$

\section{Finding $3^3{:}13$}
\label{find27}
The element $g_8$ of order $13$ lies inside a subgroup $6\udot\Suz$ of 
$2^{1+24}\udot\Co_1$.
Now $6\udot\Suz$ also lies in a (unique) subgroup $3^{1+12}\udot 3\udot \Suz$, of index $2$
in a maximal subgroup $3^{1+12}\udot 3\udot \Suz{:}2$ of the Monster.
If we can find generators for this subgroup, then we can write down
generators for a group $3^3{:}13$ containing $g_8$.

Our strategy is to find such a subgroup $6\udot\Suz$, which may be taken
to be the centralizer in the Monster of the element $za_2'$, 
and then move to the centralizer
of a suitable non-central involution, where we can find an element of order $3$
extending $3 \times 2^{1+6}\udot2\udot \U_4(2)$ to $(3^{1+4}{:}2\times 2^{1+6})\udot
\U_4(2)$, and thereby
extending $6\udot\Suz$ to $3^{1+12}\udot2\udot\Suz$. It is then easy to write down a word
for the element we want.

\subsection{The subgroup $6\udot\Suz$}
\label{6Suz}
The element
$$s_1=abababab^2abab^2ab^2$$
has order $66$, so $s_1^{22}$ is conjugate to $a_2'$. In the quotient $\Co_1$,
the elements $s_1^{22}$ and $a_2'$ generate a subgroup $A_4$, and
$a_2's_1^{22}$ conjugates $s_1^{22}$ to $a_2'$ modulo the $2$-group.
Let $s_1'$ be $s_1$ conjugated by $a_2's_1^{22}$. Then, modulo the
$2$-group, both $s_1'$ and $c_1^2$ commute with $a_2'$, and generate $3\udot\Suz$.

Hence, applying the formula, we get the following elements
centralizing $a_2'$, and generating $6\udot\Suz$:
\begin{eqnarray*}
s_1''&=&a_2's_1a_2's_1^{-1}a_2's_1\cr
s_2&=& a_2'c_1^2a_2'c_1^{-2}a_2'c_1^2
\end{eqnarray*}

\subsection{Changing post again}
\label{chpost2}
The element $$j_2=(s_2^2s_1'')^6$$ turns out to be an involution mapping to $\Co_1$-class
$2A$, and forming a $2^2$-group of Monster-type $2BAB$ with the central
involution $z$ of $2^{1+24}\udot\Co_1$. Our standard involution of this type is
$$j_0=((hi)^4i)^{15},$$ and we can check that $j_0^T$ lies in the $2^{1+24}$.
Moreover, $j_0j_2$ has order $5$, so $$j_3=(j_0j_2)^2$$ conjugates 
$j_2$ to $j_0$.

The usual brute-force approach (used once and for all)
then finds the element
$$j_4=(ab)^{27}(ab^2)^4(abab^2)^4(ababab^2abab^2)^{13}(ababab^2ab)^9
(ababab^2)^4$$
such that
$$(j_0)^{Tj_4T^{-1}}=z.$$
Hence
$$j_3Tj_4T^{-1}$$
conjugates $j_2$ to 
$z$. 
\subsection{Identifying the element of order $3$}
\label{ident3}
Writing $$j_5=((s_1'')^{22})^{j_3Tj_4T^{-1}}$$
we want to find words for the centralizer of the element $zj_5$
of order $6$. This centralizer is a group of shape
$$(2^{1+6}\times 3^{1+4}{:}2)\udot \U_4(2).$$
In particular, we want to find a non-central element of the normal $3^{1+4}$.

 We begin as usual in the $\Co_1$ quotient, and look for $3A$-elements
which commute with $j_5$. Writing
\begin{eqnarray*}
j_6&=&(ab)(ab^2)^6(ababab^2ab)^9(ababab^2)^9\cr
j_7&=&(ab)(ab^2)^8(ababab^2ab)^{11}(ababab^2)^4\cr
j_8&=&(ab)^3(ab^2)^5(ababab^2ab)^{13}(ababab^2)^6\cr
j_9&=&(ab)^3(ab^2)^{30}(ababab^2ab)^{10}(ababab^2)^4\cr
j_{10}&=&(ab)^6(ab^2)^{17}(ababab^2ab)^{12}(ababab^2)^9
\end{eqnarray*}
we have that $j_n'=(a_2')^{j_n}$ is such a $3A$-element for $n\in\{6,7,8,9,10\}$.
Moreover $$j_5'=(j_6'j_7'j_8'j_9'j_{10}')^{12}$$ is congruent to $j_5$ modulo
the $2$-group. 

To find out which of the $p_n$ and $d_n$ to multiply by, we apply the element 
$j_5(j_5')^{-1}$ to $13$ carefully selected coordinate vectors, as described in
\cite{L241}, and read off the answer from the result.
We find that the correct answer is
$$j_5''=p_1p_6p_8p_9p_{10}p_{11}p_{12}
d_1d_2d_5d_6d_{10}d_{11}d_{12}j_5'.$$
That is, $j_5''$ is actually equal to $j_5$ in the Monster.

\subsection{Finding $3^{1+12}$}
\label{3x13}
Now apply the formula so that we get generators for the centralizer
of $j_5''$ as follows. For $n\in\{6,7,8,9,10\}$, define
$$j_n''=j_5''j_n'j_5''j_n'^{-1}j_5''j_n'$$
Then $j_6''j_7''(j_8''j_9''j_{10}'')^2$ has order $10$ and we find that
$$j=(j_6''j_7''(j_8''j_9''j_{10}'')^2)^5(j_8''j_9''j_{10}''j_6''j_7''j_8''j_9''j_{10}'')^5$$
is an element in the normal $3^{1+4}$ as required.

We now know that, 
under the action of $6\udot\Suz$ generated by
$s_1''$ and $s_2$, the element $$j'=j^{Tj_4^{-1}T^{-1}j_3^{-1}}$$ and its
conjugates generate a group $3^{1+12}$.

\subsection{Finding the right $3^3$}
\label{right27}
 Moreover, the element
$g_8$ of order $13$ acts fixed-point-freely on the natural $3^{12}$ quotient
of $3^{1+12}$.
Elementary linear algebra then tells us that if we want a $3^3$ on which the
minimum polynomial of the action of $g_8$ is $x^3-x-1$, then we compute
$$\frac{x^{13}-1}{(x-1)(x^3-x-1)}=x^9+x^8-x^7+x^5-x^3-x^2-1,$$ and
hence 
(modulo the central $3$) the element
$$j''=j'^{-1}g_8^4j'g_8j'g_8j'^{-1}g_8^2j'g_8^2j'^{-1}g_8j'^{-1}g_8^2$$
is the one we want. In fact, no correction for the centre is required.

\section{Proof of the main theorem}
\label{finale}
\subsection{Analysing the $78$ cases}
\label{testcases}
The $78$ ways of extending $13$ to $D_{26}$ are obtained by taking the $6$
non-central involutions in $4A_4$, and conjugating by suitable elements of
$13^{1+2}$.
First take the involution $a_1'g_5'^3$ and check that
$w_1a_1'g_5'^3$ has order $2$, so that the $13$ conjugates 
of $a_1'g_5'^3$ by powers of $w_1'$ can be 
written
as $$(w_1')^na_1'g_5'^3$$ 
for $0\le n\le 12$.
Then conjugate these $13$ involutions by suitable elements of $\langle a_1',a_2'\rangle
\cong 2A_4$ to get the full set of $78$. For example, we may conjugate in turn by each of 
the six elements
$$1, a_2', a_2'^2, a_2'a_1', a_2'a_1'a_2', a_2'a_1'a_2'^2.$$

However, a subgroup $3\times 2\times 3$ of $13^{1+2}{:}(3\times 4A_4)$
normalizes $3^3{:}13$, and it is easy to see that it fuses the $78$ cases into
$6$ orbits, of lengths $3+3+18+18+18+18$.
These six cases are represented by the cases $(w_1')^na_1g_5'^3$ for
$n=0,1,2$, and conjugates by $a_2'a_1'$.
In each case we perform the following test. Given the fixed generators
$j'',g_8$ for $3^3{:}13$, and the $6$ involutions $x$, test each of the $13$ words
$$j''xg_8^m$$ (for $0\le m\le 12$) on a random vector
to see if it has order $3$. 
Since $j''$ is a word involving exactly $28$ occurrences of $T$ or $T^{-1}$,
and $x$ averages just $4$ such occurrences, 
each test 
involves  on average $96$ applications of $T$ or $T^{-1}$, and a slightly
larger number of applications of elements of $2^{1+24}\udot\Co_1$.
On my rather old laptop, such a test takes around $15$ minutes, and therefore
the total calculation takes around $1.5$ hours.

\subsection{Proofs} Most of the computer calculations were performed
without proof, and therefore it is necessary to provide proofs for the
few statements which we actually need in order to prove our
main theorem.

\begin{theorem}
There is no subgroup of the Monster isomorphic to $\PSL_2(27)$.
\end{theorem}
\begin{proof}
By Theorems~\ref{3313theorem} and \ref{13Atheorem}, there is a unique class
of $3^3{:}13$ which we need to consider. We prove computationally 
that the elements $j''$ and $g_8$ generate a group $3^3{:}13$,
by checking generators and relations on
two vectors whose joint stabilizer in the Monster is known to be trivial. 
Moreover, $g_8$ lies in $2^{1+24}\udot\Co_1$, so is in Monster-class $13B$.

Similarly, we check generators and relations for $(13{:}3\times 2A_4){:}2$.
Also, we check that $w_1'$ is an element of order $13$, and that it commutes
with $g_8$ and with $a_1'g_5$. Since the latter element inverts $g_8$, we deduce that
$w_1'$ is not a power of $g_8$ (or this could be checked directly). Hence
the given elements generate $13^{1+2}{:}(3\times 4A_4)$, as required.

Therefore, the test runs through all the involutions inverting $g_8$, and since
the test failed in every one of the six cases, the proof is complete. 
\end{proof}

\section*{Acknowledgements}

Thanks are due to EPSRC, whose grant, number GR/S41319, paid for the
laptop on which these calculations were performed.


\begin{thebibliography}{99}
\bibitem{Bray} J. N. Bray, { An improved method for generating the centralizer of an
involution}, {\it Arch. Math. (Basel)} {\bf 74} (2000), 241--245.

\bibitem{Atlas}
J. H. Conway, R. T. Curtis, S. P. Norton, R. A. Parker and R. A. Wilson,
{\it An Atlas of Finite Groups}, Oxford University Press, 1985.

\bibitem{S4subs} P. E. Holmes, A classification of subgroups of the
Monster isomorphic to $S_4$ and an application,
{\it J. Algebra} {\bf 319} (2008), 3089--3099.

\bibitem{2loccon} P. E. Holmes and R. A. Wilson, {A new computer construction
of the Monster using $2$-local subgroups}, {\it J. London Math. Soc.} {\bf 67} (2003), 349--364.

\bibitem{post} P. E. Holmes and R. A. Wilson,
{ A new maximal subgroup of the Monster},
{\it J. Algebra} {\bf 251} (2002), 435--447.

\bibitem{L259} P. E. Holmes and R. A. Wilson,
$\PSL_2(59)$ is a subgroup of the Monster,
{\it J. London Math. Soc.} {\bf 69} (2004), 141--152.

\bibitem{A5subs} P. E.  Holmes and R. A. Wilson, {On subgroups of the Monster
containing $A_5$'s,} {\it J. Algebra} {\bf 319} (2008), 2653--2667.

\bibitem{Mmod2} S. A. Linton, R. A. Parker, P. G. Walsh and R. A. Wilson,
Computer construction of the Monster,
{\it J. Group Theory} {\bf 1} (1998), 307--337.

\bibitem{Anatomy1} S. Norton, { Anatomy of the Monster: I,} in
{\it Proceedings of the
{\sc Atlas} Ten Years On conference (Birmingham 1995)}, pp.\ 198--214, Cambridge
Univ.\ Press, 1998.

\bibitem{Anatomy2} S. P. Norton and R. A. Wilson, Anatomy of the Monster: II,
{\it Proc. London Math. Soc.} {\bf 84} (2002), 581--598.

\bibitem{L241} S. P. Norton and R. A. Wilson,
{A correction to the $41$-structure of the Monster, a construction of a new
maximal subgroup $\L_2(41)$, and a new Moonshine
phenomenon},
(19 pages), {\it J. London Math. Soc.}, published online Jan. 2013. doi: 10.1112/jlms/jds078 

\bibitem{oddlocals} R. A. Wilson, The odd-local subgroups of the Monster,
{\it J. Austral. Math. Soc. (A)} {\bf 44} (1988), 1--16.

\bibitem{J1} R. A. Wilson, Is $J_1$ a subgroup of the Monster?,
{\it Bull. London Math. Soc.} {\bf 18} (1986), 349--350.

\bibitem{FSG} R. A. Wilson,
{\it The finite simple groups}, Springer GTM 251, 2009.


\bibitem{webatlas} R. A. Wilson et al., {\it An Atlas of Group Representations},
http://brauer.maths.qmul.ac.uk/Atlas/.

\end{thebibliography}
\end{document}